\documentclass[11pt]{amsart}

\usepackage{amsthm,amssymb,amsfonts,amsmath,enumerate}

\newcommand{\Z}{\mathbb{Z}}
\newcommand{\R}{\mathbb{R}}

\renewcommand{\phi}{\varphi}

\newcommand{\A}{{\mathbf A}}
\newcommand{\B}{{\mathbf B}}
\newcommand{\m}{{\mathbf m}}
\newcommand{\z}{{\mathbf z}}
\newcommand{\x}{{\mathbf x}}
\newcommand{\y}{{\mathbf y}}
\newcommand{\p}{{\mathbf p}}
\newcommand{\w}{{\mathbf w}}
\renewcommand{\v}{{\mathbf v}}
\newcommand{\e}{{\mathbf e}}
\newcommand{\0}{{\mathbf 0}}
\def\th{^{\text{th}}}
\newcommand{\la}{\lambda}

\newtheorem{theorem}{Theorem}

\newtheorem{lemma}[theorem]{Lemma}
\newtheorem{conjecture}[theorem]{Conjecture}

\newtheorem*{definition}{Definition}

\newcommand\comment[1]{}                
\renewcommand\comment[1]{\emph{[#1]}}           

\title{Mahonian Partition Identities Via Polyhedral Geometry}

\author{Matthias Beck}
\address{Department of Mathematics\\
         San Francisco State University\\
         San Francisco, CA 94132}
\email{beck@math.sfsu.edu}

\author{Benjamin Braun}
\address{Department of Mathematics\\
         University of Kentucky\\
         Lexington, KY 40506--0027}
\email{benjamin.braun@uky.edu}

\author{Nguyen Le}
\address{Department of Mathematics\\
         San Francisco State University\\
         San Francisco, CA 94132}
\email{nguyenle2005usa@yahoo.com}

\date{April 3, 2011}

\begin{document}

\begin{abstract}
In a series of papers, George Andrews and various coauthors successfully revitalized seemingly forgotten, powerful machinery based on MacMahon's $\Omega$ operator to systematically compute generating functions
$
  \sum_{ \la \in P } z_1^{ \la_1 } \cdots z_n^{ \la_n }
$
for some set $P$ of integer partitions $\la = (\la_1, \dots, \la_n)$. Our goal is to geometrically prove and extend many of the Andrews et al theorems, by realizing a given family of partitions as the set of integer lattice points in a certain polyhedron.
\end{abstract}

\dedicatory{Dedicated to the memory of Leon Ehrenpreis}

\thanks{We thank Carla Savage for pointing out several results in the literature that were relevant to our project.
This research was partially supported by the NSF through grants DMS-0810105 (Beck), DMS-0758321 (Braun), and DGE-0841164 (Le).}


\maketitle


\section{Introduction}

In a series of papers starting with \cite{andrewspa1}, George Andrews and various coauthors successfully revitalized seemingly forgotten, powerful machinery based on MacMahon's $\Omega$ operator \cite{macmahon} to systematically compute generating functions related to various families of integer partitions.
Andrews et al's papers concern generating functions of the form
\[
  f_P (z_1, \dots, z_n) := \sum_{ \la \in P } z_1^{ \la_1 } \cdots z_n^{ \la_n }
  \qquad \text{ and } \qquad
  f_P (q) := f_P (q, \dots, q) = \sum_{ \la \in P } q^{ \la_1 + \dots + \la_n } ,
\]
for some set $P$ of \emph{partitions} $\la = (\la_1, \dots, \la_n)$; i.e., we think of the integers $\la_n \ge \dots \ge \la_1 \ge 0$ as the \emph{parts} when some integer $k$ is written as $k = \la_1 + \dots + \la_n$.
If we do not force an order onto the $\la_j$'s, we call $\la$ a \emph{composition} of $k$.
Below is a sample of some of these striking results.

\begin{theorem}[Andrews \cite{andrewsPA2}] \label{higherdiffthm}
Let
\[
  P_r := \left\{ \la : \, \sum_{ j=0 }^t (-1)^j \binom t j \la_{ k+j } \ge 0 \ \text{ \rm for } \ k \ge 1, \ 1 \le t \le r \right\} 
\]
(where we set undefined $\la_j$'s zero).
Then
\[
  f_{ P_r } (q) = \prod_{ j=1 }^\infty \frac{ 1 }{ 1 - q^{ \binom{ j+r-1 }{ r } } } \, .
\]
In words: the number of partitions of an integer $k$ satisfying the ``higher-order difference conditions" in $P_r$ equals the number of partitions of $k$ into parts that are $r$'th-order binomial coefficients.
\end{theorem}

\begin{theorem}[Andrews--Paule--Riese \cite{andrewspauleriesePA9}] \label{ngonpartitionthm} 
Let $n \ge 3$ and
\[
  \tau := \left\{ (\la_1, \dots , \la_n) \in \Z^n : \, \la_n \ge \cdots \ge \la_1 \ge 1 \ \text{ \rm and } \ \la_1 + \cdots + \la_{n-1} > \la_n \right\} ,
\]
the set of all ``$n$-gon partitions."
Then
\[
  f_{ \tau } (q) =  \frac {q^n}{(1-q)(1-q^2) \cdots (1-q^n)} - \frac{ q^{2n-2} }{ (1-q) (1- q^2) (1 - q^4) (1 - q^6) \cdots (1 - q^{2n-2}) } \, .
\]
More generally,
\begin{align*}
  f_{ \tau } (z_1, \dots, z_n) 
  &=  \frac {Z_1}{(1-Z_1)(1-Z_2) \cdots (1-Z_n)}\\
  &\qquad - \frac{ Z_1 Z_n^{n-2} }{ (1-Z_n) (1- Z_{n-1}) (1 - Z_{n-2} Z_n) (1 - Z_{n-3} Z_n^2) \cdots (1 - Z_1 Z_n^{n-2}) }
 \, ,
\end{align*}
where $Z_j := z_j z_{ j+1 } \cdots z_n \text { for } 1 \leq j \leq n$.
\end{theorem}

The composition analogue of Theorem \ref{ngonpartitionthm} was inspired by a problem of Hermite \cite[Ex.~31]{polyaszego}, which is essentially the case $n=3$ of the following.

\begin{theorem}[Andrews--Paule--Riese \cite{andrewspauleriesePA3}] \label{hermitethm} 
Let
\[
  H := \left\{ (\la_1, \dots , \la_n) \in \Z_{ >0 }^n : \, \la_1 + \dots + \widehat{ \la_j } + \dots + \la_n \ge \la_j \ \text{ \rm for all } \ 1 \le j \le n \right\} \, .
\]
Then
\[
  f_H (q) = \frac{ q^n }{ (1-q)^n } - n \, \frac{ q^{ 2n-1 } }{ (1-q)^n (1+q)^{ n-1 } } \, .
\]
\end{theorem}

A natural question is whether there exist ``full generating function" versions of Theorems \ref{higherdiffthm} and \ref{hermitethm}, in analogy with Theorem \ref{ngonpartitionthm}; we will show that such versions (Theorems \ref{higherdiffthmgeneralized} and \ref{fullhermitethm} below) follow effortlessly from our approach. (Xin \cite[Example 6.1]{xinmacmahon} previously computed a full-generating-function related to Theorem \ref{hermitethm}.)

Our main goal is to prove these theorems geometrically, and more, by realizing a given family of partitions as the set of integer lattice points in a certain polyhedron.
This approach is not new:
Pak illustrated in \cite{pakpartitionprocams,pakpartitionsurvey} how one can obtain bijective proofs by realizing when both sides of a partition identity are generating functions of lattice points in unimodular cones (which we will define below); this included most of the identities appearing in \cite{andrewsPA2}, including Theorem~\ref{higherdiffthm}.
Corteel, Savage, and Wilf \cite{corteelsavagewilf} implicitly used the extreme-ray description of a cone (see Lemma \ref{linalglemma} below) to derive product formulas for partition generating functions, including those appearing in \cite{andrewsPA2}.
Beck, Gessel, Lee, and Savage \cite{beckgesselleesavage} used triangulations of cones to extend results of Andrews, Paule, and Riese \cite{andrewspauleriesePA7} on ``symmetrically constrained compositions.''
However, we feel that each of these papers only scratched the surface of a polyhedral approach to partition identities, and we see the current paper as a further step towards a systematic study of this approach.

While the $\Omega$-operator approach to partition identities is elegant and powerful (not to mention useful in the search for such identities), we see several reasons for pursuing a geometric interpretation of these results.  
As discussed in \cite{corteelleesavage5guidelines}, partition analysis and the $\Omega$ operator are useful tools for studying partitions and compositions defined by linear constraints, which is equivalent to studying integer points in polyhedra.
An explicit geometric approach to these problems often reveals interesting connections to geometric combinatorics, such as the connections and conjectures discussed in Sections~\ref{cayleycomp} and~\ref{hypersimplices} below.
Also, one of the great appeals of partition analysis is that it is automatic; Andrews discusses this in the context of applying the $\Omega$ operator to the 4-dimensional case of lecture-hall partitions in \cite{andrewspa1}:
\begin{quote}
The point to stress here is that we have carried off the case $j=4$ with no effective combinatorial argument or knowledge.
In other words, the entire problem is reduced by Partition Analysis to the factorization of an explicit polynomial.
\end{quote}
As we hope to show, the geometric perspective can often provide a clear view of sometimes mysterious formulas that arise from the symbolic manipulation of the $\Omega$ operator.


\section{Polyhedral cones and their lattice points}\label{latticepointsection}


We use the standard abbreviation $\z^\m := z_1^{ m_1 } \cdots z_n^{ m_n }$ for two vectors $\z$ and $\m$.  
Given a subset $K$ of $\R^n$, the \emph{(integer-point) generating function} of $K$ is
\[
  \sigma_K (z_1, \dots, z_n) := \sum_{ \m \in K \cap \Z^n } \z^\m \, .
\]
We will often encounter subsets that are cones, where a \emph{(polyhedral) cone} $C$ is the intersection of finitely many (open or closed) halfspaces whose bounding hyperplanes contain the origin.
(Thus the cones appearing in this paper will not all be closed but in general partially open.)
A closed cone has the alternative description (and this equivalence is nontrivial \cite{ziegler}) as the nonnegative span of a finite set of vectors in $\R^n$, the \emph{generators} of~$C$.

An $n$-dimensional cone in $\R^n$ is \emph{simplicial} if we only need $n$ halfspaces to describe it.
All of our cones will be \emph{pointed}, i.e., they do not contain lines.
The following exercise in linear algebra shows how to switch between the generator and halfspace descriptions of a simplicial cone.

\begin{lemma}\label{linalglemma}
Let $\A$ be the inverse matrix of $\B \in \R^{ n \times n }$. Then
\[
  \left\{ \x \in \R^n : \, \A \, \x \ge \0 \right\} = \left\{ \B \, \y : \, \y \ge \0 \right\} ,
\]
where each inequality is understood componentwise.
\end{lemma}

The (integer-point) generating function of a simplicial cone $C \subset \R^n$ can be computed from first principles when $C$ is \emph{rational}, i.e., its generators can be chosen in $\Z^n$. 
A closed cone $C$ is \emph{unimodular} if its generators form a basis of $\Z^n$; for unimodular cones, which is all we will need in what follows, we have the following simple lemma (for much more general results, see, e.g., \cite[Chapter 3]{ccd}).

\begin{lemma}\label{conegenfctlemma}
Suppose $C = \sum_{ j=1 }^{ k } \R_{ \ge 0 } \v_j + \sum_{ i=k+1 }^n \R_{ >0 } \v_i$ is a unimodular cone in $\R^n$ generated by $\v_1, \dots, \v_n \in \Z^n$. 
Then
\[
  \sigma_C (z_1, \dots, z_n) = \frac{ \prod_{i=k+1}^n \z^{\v_i} }{ \prod_{ j=1 }^n \left( 1 - \z^{ \v_j } \right) } \, .
\]
\end{lemma}







\section{Unimodular Cones}

Recall from Theorem \ref{higherdiffthm} that
\[
  P_r = \left\{ \la : \, \sum_{ j=0 }^t (-1)^j \binom t j \la_{ k+j } \ge 0 \ \text{ \rm for } \ k \ge 1, \ 1 \le t \le r \right\} 
\]
(where we set undefined $\la_j$'s zero).
Let
\[
  P_r^n := \left\{ (\la_1, \dots , \la_n) \in \Z^n : \, \sum_{ j=0 }^t (-1)^j \binom t j \la_{ k+j } \ge 0 \ \text{ \rm for } \ 1 \le k \le n, \ 1 \le t \le r \right\} 
\]
consist of all partitions in $P_r$ with at most $n$ parts.
As a warm-up example we will compute the (full) generating function of $P_r^n$:

\begin{theorem}\label{higherdiffthmgeneralized}
\begin{align*}
  &f_{ P_r^n } (z_1, \dots, z_n) = \\
  &\frac{ 1 }{ \left( 1 - z_1 \right) \left( 1 - z_1^r z_2  \right) \left( 1 - z_1^{ \binom{ r+1 }{ r-1 } } z_2^r z_3 \right) \left( 1 - z_1^{ \binom{ r+2 }{ r-1 } } z_2^{ \binom{ r+1 }{ r-1 } } z_3^r z_4 \right) \cdots \left( 1 - z_1^{ \binom{ r+n-2 }{ r-1 } } z_2^{ \binom{ r+n-3 }{ r-1 } } \cdots z_{ n-1 }^r z_n \right) } \, .
\end{align*}
\end{theorem}

\noindent
Note that Theorem \ref{higherdiffthm} follows upon setting $z_1 = \dots = z_n = q$, using the identity
\[
  \binom{ r+j-2 }{ r-1 } + \binom{ r+j-3 }{ r-1 } + \dots + r + 1 = \binom{ r+j-1 }{ r } \, ,
\]
and taking $n \to \infty$.

\begin{proof}
It is easy to see that the inequalities
\[
  \sum_{ j=0 }^t (-1)^j \binom t j \la_{ k+j } \ge 0 \ \text{ \rm for } \ 1 \le k \le n, \ 1 \le t \le r \, ,
\]
which define $P_r^n$, are implied by the inequalities for $t=r$.
Thus the cone containing $P_r^n$ as its integer lattice points is
\begin{align*}
  K &:= \left\{ (x_1, \dots , x_n) \in \R^n : \, \sum_{ j=0 }^r (-1)^j \binom r j x_{ k+j } \ge 0 \ \text{ \rm for } \ 1 \le k \le n \right\} \\
    &= \left\{ 
  \begin{bmatrix}
  1 & r & \binom{ r+1 }{ r-1 } & \binom{ r+2 }{ r-1 } & \cdots & \binom{ r+n-2 }{ r-1 } \\
  0 & 1 & r & \binom{ r+1 }{ r-1 } & \cdots & \binom{ r+n-3 }{ r-1 } \\
  0 & 0 & 1 & r & \cdots & \binom{ r+n-4 }{ r-1 } \\
  \vdots & & \ddots & \ddots & \ddots &\vdots  \ \\ 
  0 & & & 0 & 1 & r \\
  0 &  & \cdots & & 0 &1 \ 
  \end{bmatrix}
  \y : \, y_1, \dots, y_n \ge 0 \right\} 
\end{align*}
(whose generators we can compute, e.g., with the help of Lemma \ref{linalglemma}).
Thus $K$ is unimodular and, by Lemma~\ref{conegenfctlemma},
\begin{align*}
  &\sigma_K (z_1, \dots, z_n) = \\
  &\frac{ 1 }{ \left( 1 - z_1 \right) \left( 1 - z_1^r z_2  \right) \left( 1 - z_1^{ \binom{ r+1 }{ r-1 } } z_2^r z_3 \right) \left( 1 - z_1^{ \binom{ r+2 }{ r-1 } } z_2^{ \binom{ r+1 }{ r-1 } } z_3^r z_4 \right) \cdots \left( 1 - z_1^{ \binom{ r+n-2 }{ r-1 } } z_2^{ \binom{ r+n-3 }{ r-1 } } \cdots z_{ n-1 }^r z_n \right) } \, .
\end{align*}
\end{proof}

The idea behind this approach towards Theorem \ref{higherdiffthm} can be found, in disguised form, in \cite{corteelsavagewilf} and \cite{pakpartitionprocams}.
See also \cite{canfieldcorteelhitczenko,corteelsavageinequalities} for bijective approaches to Theorem \ref{higherdiffthm} and its asymptotic consequences.
We included this proof here in the interest of a self-contained exposition and also because none of \cite{andrewsPA2,corteelsavagewilf,pakpartitionprocams} contains a full generating function version of (analogues of) Theorem~\ref{higherdiffthm}.


\section{Differences of two cones}

They key idea behind the proof of Theorem~\ref{ngonpartitionthm} is to observe that the non-simplicial cone
\[
  K := \left\{ (x_1, \dots , x_n) \in \R^n : \, x_n \ge \cdots \ge x_1 > 0 \ \text{ and } \ x_1 + \cdots + x_{n-1} > x_n \right\} ,
\]
whose integer lattice points form Andrews--Paule--Riese's set $\tau$ of $n$-gon partitions, can be written as a difference $K = K_1 \setminus K_2$ of two simplicial cones.
Specifically, set
\begin{align*}
  K_1 &:= \left\{ (x_1, \dots , x_n) \in \R^n : \, x_n \ge \cdots \ge x_1 > 0 \right\} \\
    &= \left\{ 
  \begin{bmatrix}
  1&0&0&\cdots&\ 0 \quad &0 \ \\
  1&1&0&\cdots&\ 0 \quad &0 \ \\
  1&1&1&\cdots&\ 0 \quad &0 \ \\ 
  \vdots&\vdots&\vdots&\ddots&\ \vdots \quad &\vdots  \ \\ 
  1&1&1&\cdots&\ 1 \quad &0 \ \\
  1&1&1&\cdots&\ 1 \quad &1 \ 
  \end{bmatrix}
  \y : \begin{array}{l}
    y_1 > 0 \, , \\
    y_2, \dots, y_n \ge 0
  \end{array} \right\} 
\end{align*}
and
\begin{align*}
  K_2 &:= \left\{ (x_1, \dots , x_n) \in \R^n : \, x_n \ge \cdots \ge x_1 > 0 \ \text{ and } \ x_1 + \cdots + x_{n-1} \le x_n \right\} \\
  &= \left\{ 
  \begin{bmatrix}
  1&0&0&\cdots&\ 0 \quad &0 \ \\
  1&1&0&\cdots&\ 0 \quad &0 \ \\
  1&1&1&\cdots&\ 0 \quad &0 \ \\ 
  \vdots&\vdots&\vdots&\ddots&\ \vdots \quad &\vdots  \ \\ 
  1&1&1&\cdots&\ 1 \quad &0 \ \\
  n-1&n-2&n-3&\cdots&\ 1 \quad &1 \ 
  \end{bmatrix}
  \y : \begin{array}{l}
    y_1 > 0 \, , \\
    y_2, \dots, y_n \ge 0
  \end{array} \right\} 
\end{align*}
(whose generators we can compute, e.g., with the help of Lemma \ref{linalglemma}).
One can see immediately from the generator matrices that both $K_1$ and $K_2$ are unimodular.
(In a geometric sense, this is suggested by the form of the identity in Theorem \ref{ngonpartitionthm}. A similar simplification-through-taking-differences phenomenon is described in the fifth ``guideline" of Corteel, Lee, and Savage \cite{corteelleesavage5guidelines}, which inspired our proof.)
By Lemma~\ref{conegenfctlemma}
\[
  \sigma_{K_1} (z_1, \dots, z_n) = \frac{ z_1 \cdots z_n }{ (1 - z_n) (1 - z_{ n-1 } z_n) \cdots (1 - z_1 \cdots z_n)}
\]
and
\begin{align*}
  &\sigma_{K_2} (z_1, \dots, z_n) \\
  &\qquad = \frac{ z_1 \cdots z_{ n-1 } z_n^{ n-1 } }{ \left(1 - z_1 \cdots z_{ n-1 } z_n^{ n-1 } \right) \left(1 - z_2 \cdots z_{ n-1 } z_n^{ n-2 } \right) \left(1 - z_3 \cdots z_{ n-1 } z_n^{ n-3 } \right) \cdots \left(1 - z_{ n-1 } z_n \right) \left(1 - z_n \right) } \\
  &\qquad = \frac{ Z_1 Z_n^{n-2} }{ (1-Z_n) (1- Z_{n-1}) (1 - Z_{n-2} Z_n) (1 - Z_{n-3} Z_n^2) \cdots (1 - Z_1 Z_n^{n-2}) } \, ,
\end{align*}
and the identity $\sigma_K (z_1, \dots, z_n) = \sigma_{K_1} (z_1, \dots, z_n) - \sigma_{K_2} (z_1, \dots, z_n)$ completes the proof.
\hfil $\Box$


\section{Differences of multiple cones}

The ``cone behind" Theorem~\ref{hermitethm} is
\[
  K := \left\{ (x_1, \dots , x_n) \in \R_{ >0 }^n : \, x_j \le x_1 + \dots + \widehat{ x_j } + \dots + x_n \ \text{ for all } \ 1 \le j \le n \right\} ;
\]
Theorem \ref{hermitethm} follows from the following result upon setting $z_1 = \dots = z_n = q$.

\begin{theorem}\label{fullhermitethm}
\[
  \sigma_K (z_1, \dots, z_n) = \frac{ z_1 \cdots z_n }{ (1-z_1) \cdots (1-z_n) } - \sum_{ k=1 }^n \frac{ z_1 \cdots z_{ k-1 } z_k^n z_{ k+1 } \cdots z_n }{ (1-z_k) \prod_{ {j=1} \atop {j \ne k} }^n (1 - z_k z_j) } \, .
\]
\end{theorem}

\begin{proof}
Let $\e_j$ denote the $j$th unit vector in $\R^n$.
Observe that the non-simplicial cone $K$ is expressible as a difference 
$
  K = O \setminus \bigcup_{ k=1 }^n C_k \, ,
$
where $O := \sum_{ j=1 }^n \R_{ >0 } \, \e_j$ and $C_k$ is the cone
\begin{align*}
C_k &:= \left\{ (x_1, \dots , x_n) \in \R_{ >0 }^n : \, x_k > x_1 + \dots + \widehat{ x_j } + \dots + x_n  \right\} \\
&= \R_{ >0 } \, \e_k + \sum_{ {j=1} \atop {j \ne k} }^n \R_{ >0 } \left( \e_j + \e_k \right) \\
\end{align*}
Note that if $i\neq j$, then $C_i\cap C_j=\emptyset$.
Thus, the closure of $K$ is ``almost'' the positive orthant $O$, except that we have to exclude points in $O$ that can only be written as a linear combination that requires a single $\e_k$ (as opposed to a linear combination of the vectors $\e_j + \e_k$).
(A similar simplification-through-taking-differences phenomenon appeared in the original proof of Theorem \ref{hermitethm}.)
In generating-function terms, this set difference gives, by Lemma~\ref{conegenfctlemma},
\begin{align*}
  \sigma_K (z_1, \dots, z_n)
  &= \sigma_O (z_1, \dots, z_n) - \sum_{ k=1 }^n \sigma_{ C_k } (z_1, \dots, z_n) \\
  &= \frac{ z_1 \cdots z_n }{ (1-z_1) \cdots (1-z_n) } - \sum_{ k=1 }^n \frac{ z_1 \cdots z_{ k-1 } z_k^n z_{ k+1 } \cdots z_n }{ (1-z_k) \prod_{ {j=1} \atop {j \ne k} }^n (1 - z_k z_j) } \, . \qedhere
\end{align*}
\end{proof}

Three remarks on this theorem are in order. 
First, as already mentioned, Xin \cite[Example 6.1]{xinmacmahon} previously computed a different full-generating-function related to Theorem \ref{hermitethm}; Xin's generating function handles non-negative, rather than positive, $k$-gon partitions.
Second, the cone $K$ is related to the second hypersimplex, a well-known object in geometric combinatorics (see Section~\ref{hypersimplices} for more details).

Third, $K$ is a suitable candidate for the ``symmetrically constrained" approach in \cite{beckgesselleesavage}; however, one should expect that this approach would give a different form for the generating function $\sigma_K (z_1, \dots, z_n)$ from the one given in Theorem~\ref{fullhermitethm}.
The symmetrically constrained approach produces a triangulation of the cone $K$ that is invariant under permutation of the standard basis vectors in $\R^n$, and then uses this triangulation to express $\sigma_K(z_1,\dots,z_n)$ as a positive sum of rational generating functions for these cones (after some geometric shifting).
The terms in this sum will all have $\frac{1}{1-z_1 z_2 \cdots z_n}$ as a factor, as each of the simplicial cones in the triangulation of $K$ will have the all-ones vector as a ray generator; this will clearly produce a different form from that in Theorem~\ref{fullhermitethm}.



\section{Cayley Compositions}\label{cayleycomp}

A \emph{Cayley composition} is a composition $\la=\left( \la_1, \dots, \la_{j-1} \right)$ that satisfies $1 \le \la_1 \le 2$ and $1 \le \la_{ i+1 } \le 2 \la_i$ for $1 \le i \le j-2$.
Thus, the Cayley compositions with $j-1$ parts are precisely the integer points in
\[
  C_{j} := \left\{ (\la_1, \dots , \la_{j-1}) \in \Z_{ >0 }^{j-1} : \, \la_1 \leq 2 \ \text{ \rm and } \ \la_i \leq 2\la_{i-1} \ \text{ \rm for all } \ 2 \le i \le j-1 \right\} \, .
\]
Our apparent shift in indexing maintains continuity between our statements and \cite{andrewspauleriesestrehlPA5}, where Cayley compositions always begin with a $\la_0=1$ part.
Let $f_{C_j} \left( z_1, \dots, z_{j-1} \right)$ be the generating function for $C_j$.
The following theorem is quite surprising.

\begin{theorem}[Andrews--Paule--Riese--Strehl \cite{andrewspauleriesestrehlPA5}] \label{cayley} 
Let
\[
  C_{j} := \left\{ (\la_1, \dots , \la_{j-1}) \in \Z_{ >0 }^{j-1} : \, \la_1 \leq 2 \ \text{ \rm and } \ \la_i \leq 2\la_{i-1} \ \text{ \rm for all } \ 2 \le i \le j-1 \right\} \, .
\]
Then for $j\geq 2$,
\begin{align*}
f_{C_{j}}(1,1,\ldots,1,q) = & \sum_{h=1}^{j-2}\frac{b_{j-h-1}(-1)^{h-1}q^{2^h-1}}{(1-q)(1-q^2)(1-q^4)\cdots(1-q^{2^{h-1}})} \\
&
 +  \frac{(-1)^jq^{2^{j-1}-1}(1-q^{2^{j-1}})}{(1-q)(1-q^2)(1-q^4)\cdots(1-q^{2^{j-2}})}
\end{align*}
where $b_k$ is the coefficient of $q^{2^k-1}$ in the power series expansion of 
\[
\frac{1}{1-q}\prod_{m=0}^\infty\frac{1}{1-q^{2^m}} 
\, .
\]
\end{theorem}

Theorem~\ref{cayley} is derived as a consequence of the following recurrence relation obtained via MacMahon's $\Omega$ calculus.

\begin{theorem}[Andrews--Paule--Riese--Strehl \cite{andrewspauleriesestrehlPA5}]\label{recurrence}
\[ 
f_{C_{j}} \left( z_1, \dots, z_{j-1} \right) = \frac{ z_{j-1} }{ 1 - z_{j-1} } \left( f_{ C_{j-1} } \left( z_1, \dots, z_{ j-2 } \right) - f_{ C_{j-1} } \left( z_1, \dots, z_{ j-3 }, z_{ j-2 } z_{j-1}^2 \right) \right) .
\]
\end{theorem}

Once this formula is obtained, the proof of Theorem~\ref{cayley} in \cite{andrewspauleriesestrehlPA5} proceeds by repeatedly iterating the recurrence, specialized to $f_{C_j}(1,\ldots,1,q)$.
The final step is to argue that the sum of rational functions in Theorem~\ref{cayley}, as analytic functions, must exhibit cancellation.
We remark that Corteel, Lee, and Savage \cite[Section 3]{corteelleesavage5guidelines} gave an alternative proof of Theorem~\ref{recurrence}.

Via geometry, we can shed light on the initial recurrence relation from three perspectives.
First, we recognize that the recurrence reflects expressing $C_j$ as a difference of two subspaces of $\R^{j-1}$ defined by linear constraints.

\begin{proof}[First proof of Theorem~\ref{recurrence}]
As a subspace of $\R^{j-1}$, $C_j=K_{1,j}\setminus K_{2,j}$ where
\[K_{1,j}:=\left\{ (x_1, \dots , x_{j-1}) \in \R^{j-1} : \, 1\leq x_1 \le 2, \ 1\leq x_{ i+1 } \le 2 x_i \ \text{ for } \ 1 \le i \le j-3, \text{ and } 1\leq x_{j-1} \right\}\]
and
\[K_{2,j}:=\left\{ (x_1, \dots , x_{j-1}) \in \R^{j-1} : \, 1\leq x_1 \le 2, \ x_{ i+1 } \le 2 x_i \, \text{ for } \, 1 \le i \le j-3, \ x_{j-1} > 2 x_{ j-2 } \right\} \, .\]
If we distribute the leading multiplier in the right-hand side of the recurrence for $f_{C_j}$, the first term is the generating function of $K_{1,j}$, as there are no restrictions on the size of $x_{j-1}$.
On the other hand, the integer points $\m \in K_{2,j}$ are precisely those in $K_{1,j}$ satisfying $x_{j-1}>2x_{j-2}$, which is equivalent to the condition that $\z^\m$ be divisible by $z_{j-2}z_{j-1}^2$.
The second term of the recurrence records precisely these integer points.
\end{proof}

Our second proof amounts to a simple observation regarding the integer-point transform of $C_j$.

\begin{proof}[Second proof of Theorem~\ref{recurrence}]
Since for any $\la\in C_j\cap \Z^{j-1}$ we have $1\leq \la_{j-1}\leq 2\la_{j-2}$,
\begin{align*}
f_{C_{j}}(z_1,\ldots,z_{j-1}) 
&= \sum_{\la\in C_{j}\cap \Z^{j-1}}\z^\la \\
&= \sum_{\la\in C_{j-1}\cap \Z^{j-2}} \z^\la(z_{j-1}+z_{j-1}^2+\cdots+z_{j-1}^{2\la_{j-2}}) \\
&= z_{j-1} \sum_{\la \in C_{j-1}\cap \Z^{j-2}} \z^\la \frac{1-z_{j-1}^{2\la_{j-2}}}{1-z_{j-1}} \\
&= \frac{z_{j-1}}{1-z_{j-1}}\sum_{\la \in C_{j-1}\cap \Z^{j-2}} \z^\la - \z^\la z_{j-1}^{2\la_{j-2}} \\
&= \frac{z_{j-1}}{1-z_{j-1}} \left( f_{C_{j-1}}(z_1,\ldots,z_{j-1})-f_{C_{j-1}}(z_1,\ldots,z_{j-3},z_{j-2}z_{j-1}^2) \right) . \qedhere
\end{align*}
\end{proof}

Following their statement of Theorem~\ref{cayley}, the authors of \cite{andrewspauleriesestrehlPA5} make the following comment:
\begin{quote}
It hardly needs to be pointed out that [this formula] is a surprising representation of a polynomial.
Indeed, the right-hand side does not look like a polynomial at all.
\end{quote}
Such a statement suggests that Brion's formula \cite{brion} for rational polytopes is lurking in the background; our third proof of Theorem~\ref{recurrence} is based on this formula.
Given a rational convex polytope $P$, we first define the \emph{tangent cone} at a vertex $\v$ of $P$ to be
\[
  T_P(\v) := \left\{ \v+\alpha (\p-\v) : \, \alpha \in \R_{ \ge 0 } , \, \p \in P \right\} .
\]
\begin{theorem}[Brion]
Suppose $P$ is a rational convex polytope.
Then we have the following identity of rational generating functions:
\[
\sigma_{P}(\z)=\sum_{\v \text{ \rm a vertex of } P}\sigma_{T_P(\v)}(\z) \, .
\]
\end{theorem}
Note that the sum on the right-hand side is a sum of rational functions, while the left-hand side yields a polynomial.

\begin{proof}[Third proof of Theorem~\ref{recurrence}]
To interpret the recurrence as a consequence of Brion's formula, we first assume that the $f_{C_{j-1}}$'s are expressed in the form of the right-hand side of Brion's formula, i.e., as a sum of integer-point transforms of the tangent cones at the vertices of $C_{j-1}$.
We next rewrite the recurrence as
\[ 
f_{C_{j}} \left( z_1, \dots, z_{j-1} \right) = \frac{ z_{j-1} }{ 1 - z_{j-1} } \, f_{ C_{j-1} } \left( z_1, \dots, z_{ j-2 } \right) + \frac{1}{1 - z_{j-1}^{-1}} \, f_{ C_{j-1} } \left( z_1, \dots, z_{ j-3 }, z_{ j-2 } z_{j-1}^2 \right) .
\]
The polytope $C_j$ is a combinatorial cube; this can be easily seen by induction on $j$ after observing that in $C_{j-1}\times \R$ the hyperplanes $x_{j-1}=1$ and $x_{j-1}=2x_{j-2}$ do not intersect.
Thus, the tangent cones for vertices of $C_j$ can be expressed in terms of the tangent cones for vertices of $C_{j-1}$.
Given a vertex $\v=\{v_1,\ldots,v_{j-2}\}$ of $C_{j-1}$, the two vertices of $C_j$ obtained from $\v$ are $(\v,1)$ and $(\v,2v_{j-2})$.
For the vertex $(\v,1)$ in $C_j$, it is immediate that 
\[
\sigma_{T_{C_{j-1}}((\v,1))}(\z)=\frac{1}{1-z_{j-1}}\, \sigma_{T_{C_{j-2}}(\v)}(\z) \, .
\]
Our proof will be complete after we show that for the vertex $(\v,2v_{j-2})$ in $C_j$,
\[
\sigma_{T_{C_{j-1}}((\v,2v_{j-2}))}(\z)=\frac{1}{1-z_{j-1}^{-1}}\, \sigma_{T_{C_{j-2}}(\v)}(z_1,\ldots,z_{j-3},z_{j-2}z_{j-1}^2) \, .
\]
This follows from the fact that the edges in $C_j$ emanating from $(\v,2v_{j-2})$ terminate in the vertex $(\v,1)$ and in the vertices $(\w,2w_{j-2})$ for vertices $\w$ of $C_{j-1}$ that are connected to $\v$ by an edge in $C_{j-1}$.
Thus, Theorem~\ref{recurrence} follows from Brion's formula and induction.
\end{proof}

There is an interesting remark about Theorem~\ref{cayley} and Brion's formula; while one might hope that the expression in Theorem~\ref{cayley} is obtained by directly specializing Brion's formula to $z_1=\cdots=z_{j-2}=1$ and $z_{j-1}=q$, this is not the case.
This specialization is not actually possible, as some of the rational functions for tangent cones in $C_j$ have denominators that lack a $z_{j-1}$ variable, and hence this specialization would require evaluating rational functions at poles.
The authors of \cite{andrewspauleriesestrehlPA5} use the recurrence in Theorem~\ref{recurrence} in a more subtle way, in that they first specialize the recurrence to 
\[ 
f_{C_{j}} \left( 1, \dots, 1,q \right) = \frac{ q }{ 1 - q } \left( f_{ C_{j-1} } \left( 1, \dots, 1 \right) - f_{ C_{j-1} } \left( 1, \dots, 1, q^2 \right) \right) 
\]
and then iterate the recurrence.
In doing this, they simultaneously use the interpretation of $f_{C_j}(\z)$ as a polynomial (for the all-ones specialization) and also the interpretation of $f_{C_j}(\z)$ as a rational function (for the specialization involving $q^2$).
Thus, while Theorem~\ref{cayley} looks similar to a Brion-type result, it is obtained differently.
We remark that by specializing $z_1=\cdots=z_{j-1}=q$ in Brion's formula for $C_j$, one would obtain a representation of the polynomial $f_{C_j}(q,\ldots,q)$ as a sum of rational functions of~$q$.


\section{Directions for further investigation}\label{hypersimplices}

\subsection{Cones over hypersimplices}

We can view the cone $K$ of the previous section as a cone over a ``half-open'' version of the second \emph{hypersimplex}
\[
  \Delta(2,n) := \left\{ (x_1,\ldots,x_n)\in [0,1]^n : \, \sum_{i=1}^nx_i=2 \right\} ,
\]
in the following manner.
The linear inequality $x_j\geq x_1+\cdots +\widehat{x_j}+ \cdots + x_n$ is equivalent to $\frac{\sum_{i=1}^nx_i}{2}\leq x_j$. When $\sum_{i=1}^nx_i=1$, we are considering the ``slice'' of $K$ that is constrained by $0<x_j\leq \frac{1}{2}$ and $\sum_{i=1}^nx_i=1$, which is $\frac{1}{2}$ of $\Delta(2,n)$ with the condition that $0<x_j$ for all $j$.
From this perspective, we can view the $n$-gon compositions of $t$ as 
\begin{align*}
  H(t)
  &:= \left\{ (\la_1, \dots , \la_n) \in \Z_{ \ge 0 }^n :
  \begin{array}{l}
    \la_1 + \dots + \la_n = t \, , \\
    \la_j \le \la_1 + \dots + \widehat{ \la_j } + \dots + \la_n \ \text{ for all } 1 \le j \le n
  \end{array} \right\} \\
  &= \left\{ (\la_1, \dots , \la_n) \in \Z^n :
  \begin{array}{l}
    \la_1 + \dots + \la_n = t \, , \\
    0 \le \la_j \le \frac t 2 \ \text{ for all } 1 \le j \le n
  \end{array} \right\} .
\end{align*}
The second hypersimplex is a well-studied object; for example, in matroid theory $\Delta(2,n)$ is the matroid basis polytope for the $2$-uniform matroid on $n$ vertices, while in combinatorial commutative algebra $\Delta(2,n)$ is the subject of~\cite[Chapter 9]{sturmfels}.

It would be interesting to consider analogues of Theorem~\ref{hermitethm} for the general case of the $k\th$ hypersimplex $\Delta(k,n):=\{(x_1,\ldots,x_n)\in [0,1]^n : \sum_{i=1}^nx_i=k\}$.
The associated composition counting function has a natural interpretation: in
\[
  \left\{ (\la_1, \dots , \la_n) \in \Z^n :
  \begin{array}{l}
    \la_1 + \dots + \la_n = t \, , \\
    0 \le \la_j \le \frac t k \ \text{ for all } 1 \le j \le n
  \end{array} \right\}
\]
are all compositions of $t$ whose parts are at most $\frac t k$ (i.e., the parts are not allowed to be too large, where ``too large" depends on~$k$).


\subsection{Cayley polytopes}

We refer to the polytopes $C_j$ from Section~\ref{cayleycomp} as \emph{Cayley polytopes}.
By taking a geometric view of Cayley compositions as integer points in $C_j$, we may shift our focus from combinatorial properties of the integer points to properties of $C_j$ itself.
Recall that the \emph{normalized volume} of $C_j$ is
\[
  \mathrm{Vol}(C_j):=(j-1)! \, \mathrm{vol}(C_j) \, ,
\]
where $\mathrm{vol}(C_j)$ is the Euclidean volume of $C_j$.
Based on experimental data obtained using the software LattE \cite{koeppelatte} and the Online Encyclopedia of Integer Sequences \cite{sloaneonlineseq}, we make the following conjecture:

\begin{conjecture}\label{cayleyconj}
For $j\geq 2$, $\mathrm{Vol}(C_j)$ is equal to the number of labeled connected graphs on $j-1$ vertices.\footnote{
After a preprint of the current article was made public, Matjaz Konvalinka and Igor Pak communicated to us that they resolved Conjecture \ref{cayleyconj} by a direct combinatorial argument.
}
\end{conjecture}


\bibliographystyle{amsplain}
\bibliography{bib}

\end{document}